\documentclass[12pt]{article}

\usepackage{enumerate}
\usepackage{hyperref}
\usepackage{bm}
\usepackage{comment}
\usepackage{amsfonts}
\usepackage{graphicx,caption}
\usepackage{bm}
\usepackage{amsmath, amsthm, amssymb}
\usepackage{graphicx}
\usepackage{hyperref}
 \usepackage{relsize}
 \usepackage{algpseudocode}
\usepackage[Algorithm]{algorithm}
\algnewcommand{\IIf}[1]{\State\algorithmicif\ #1\ \algorithmicthen}
\algnewcommand{\EndIIf}{\unskip\ \algorithmicend\ \algorithmicif}

\algnewcommand{\FFor}[1]{\State\algorithmicfor\ #1\ \algorithmicdo}
\algnewcommand{\EndFFor}{\unskip\ \algorithmicend\ \algorithmicfor}

\usepackage{bbm}

\DeclareMathOperator{\Aut}{Aut}
\DeclareMathOperator{\inte}{int}
\DeclareMathOperator{\Ind}{Ind}

\usepackage[margin=2cm]{geometry}

\newtheorem{theorem}{Theorem}[section]
\newtheorem{lemma}[theorem]{Lemma}

\newtheorem{claim}[theorem]{Claim}
\newtheorem{proposition}[theorem]{Proposition}
\newtheorem{problem}[theorem]{Problem}

\theoremstyle{definition}

\theoremstyle{remark}
\newtheorem{remark}[theorem]{Remark}

\numberwithin{equation}{section}

\begin{document}

\title{The local weak limit of $k$-dimensional hypertrees}


\author{Andr\'as M\'esz\'aros}
\date{}




\maketitle

\begin{abstract}
Let $\mathcal{C}(n,k)$ be the set of $k$-dimensional simplicial complexes $C$ over a fixed set of $n$
vertices such that:

\begin{enumerate}[\hspace{0.3cm}(1)]
\item $C$ has a complete $k-1$-skeleton;
\item $C$ has precisely ${{n-1}\choose {k}}$ $k$-faces;
\item the homology group $H_{k-1}(C)$ is finite. 
\end{enumerate} 

Consider the probability measure on $\mathcal{C}(n,k)$ where the probability of a simplicial complex $C$ is proportional to $|H_{k-1}(C)|^2$. For any fixed $k$, we determine the local weak limit of these random simplicial complexes as $n$ tends to infinity. 

This local weak limit turns out to be the same as the local weak limit of  the $1$-out $k$-complexes investigated by Linial and Peled.

\end{abstract}

\section{Introduction}

Motivated by the results of Kalai \cite{kalai1983enumeration} and Lyons \cite{lyons2009random}, we consider the probability measure $\nu_{n,k}$ on the set of $k$-dimensional hypertrees over a fixed set of $n$ vertices where the probability of a hypertree $C$ is proportional to $|H_{k-1}(C)|^2$. Let~$C_{n,k}$ be a random hypertree of law $\nu_{n,k}$. The random bipartite graph $G_{n,k}$ is defined as the Hasse diagram of $C_{n,k}$ restricted to the faces of dimension $k$ and $k-1$.

The main result of this paper is the following.
\begin{theorem}\label{thmMain}
For a fixed $k$, the local weak limit of the graphs $G_{n,k}$ is the semi-$k$-ary skeleton tree.
\end{theorem} 

Now we elaborate on the notions mentioned above.

 To motivate the definition of hypertrees, we first explain the higher dimensional generalization of Cayley's formula by Kalai \cite{kalai1983enumeration}. Recall that Cayley's formula states that the number of spanning trees of the complete graph on the vertex set $[n]=\{1,2,\dots,n\}$ is equal to 
\[n^{n-2}.\]
This formula was first proved by Borchardt \cite{borchardt1861interpolationsformel}, and it was stated using graph theoretic terms by Cayley \cite{cayley1889theorem}. This theorem has several proofs \cite{aigner2010proofs}. One of the most well-known proofs is based on the {Cauchy-Binet} formula. We briefly recall this proof, since this gives us a good insight why the notions in Kalai's generalization are natural. 
 
Let us consider the complete graph on the vertex set $[n]$. Its oriented incidence matrix $I_{n,1}$ is a matrix  where the rows are indexed  with the vertex set $[n]$, the columns are indexed with the edge set ${{[n]}\choose{2}}$, and the entries are defined as follows. For a vertex $v$ and an edge $e=\{i,j\}$ such that $i<j$, we set
\[I_{n,1}(v,e)=\begin{cases}
-1 &\text{ if $v=i$,}\\
+1&\text{ if $v=j$,}\\
 0&\text{ otherwise.}
\end{cases}
\] 
Note that the rows of $I_{n,k}$ are not linearly independent, because their sum is zero. Thus, it will be useful to consider the matrix $\hat{I}_{n,1}$ which is obtained from $I_{n,1}$ by deleting the row corresponding to the vertex $n$. 
Given a subset $F$ of the edges, let $\hat{I}_{n,1}[F]$ be the submatrix of $\hat{I}_{n,1}$ determined by the columns corresponding to the elements of $F$. 
\begin{proposition}\label{propfa}
Let $F$ be an $n-1$ element subset of the edges. Then
\[|\det \hat{I}_{n,1}[F]|=\begin{cases} 
1&\text{if $F$ is a spanning tree,}\\
0&\text{otherwise.}
\end{cases}\]
\end{proposition} 
Combining this proposition with the Cauchy-Binet formula, we obtain that
\[\det (\hat{I}_{n,1} \hat{I}_{n,1}^T)=\sum_{F} |\det \hat{I}_{n,1}[F]|^2=|\{ F |\text{ $F$ is a spanning tree}\}|,\]
where the summation is over all the $n-1$ element subsets of the edges. Here, $\hat{I}_{n,1} \hat{I}_{n,1}^T$ is actually equal to the reduced Laplacian of the complete graph. One can evaluate the determinant of this matrix to obtain Cayley's formula.

Next, we try to find a suitable generalization of  Cayley's formula. First, we define the higher dimensional analogue of the oriented incidence matrix $I_{n,1}$ of the complete graph. For $k\ge 1$, let $I_{n,k}$ be a matrix indexed by ${{[n]}\choose {k}}\times {{[n]}\choose{k+1}}$ defined as follows.\footnote{We use the notation ${{[n]}\choose{k}}$ for $\{X\subset [n]|\quad|X|=k\}$.} Let $X=\{x_0,x_1,\dots,x_k\}\subset [n]$ such that $x_0<x_1<\dots<x_k$. For a $Y\in {{[n]}\choose {k}}$, we set
\[I_{n,k}(Y,X)=\begin{cases}
(-1)^i&\text{if }Y=X\backslash\{x_i\},\\
0&\text{otherwise.} 
\end{cases}
\]   
In other words, $I_{n,k}$ is just the matrix of the $k$th boundary operator of the simplex on the vertex set $[n]$. Note that there are several linear dependencies among the rows of $I_{n,k}$. 
Let us consider the submatrix $\hat{I}_{n,k}$ of $I_{n,k}$ determined by the rows indexed by the elements of ${{[n-1]}\choose {k}}$. One can prove that the rows of $\hat{I}_{n,k}$ form a bases of the rowspace of $I_{n,k}$. For a $k$-dimensional simplicial complex $C$ on the vertex set $[n]$, the submatrix of $\hat{I}_{n,k}$ determined by the columns corresponding to the $k$-faces of $C$ is denoted by $\hat{I}_{n,k}[C]$.  

Motivated by Proposition \ref{propfa}, a $k$-dimensional simplicial complex $C$ on the vertex set $[n]$ is called a $k$-dimensional hypertree, if the number of $k$-faces of $C$ is equal to ${{n-1}\choose{k}}$  and $\det \hat{I}_{n,k}[C]\neq 0$. One can see that a $k$-dimensional hypertree must have a complete  $k-1$-skeleton. Note that for $k>1$, it is no longer true that $|\det \hat{I}_{n,k}[C]|=1$ for a hypertree~$C$. In fact, $|\det \hat{I}_{n,k}[C]|$ has homology theoretic relevance. Namely, consider a $k$-dimensional simplicial complex $C$ over the vertex set $[n]$ with a complete $k-1$ skeleton and ${{n-1}\choose{k}}$ $k$-faces. If~$\det \hat{I}_{n,k}[C]\neq 0$, then $|H_{k-1}(C)|=|\det \hat{I}_{n,k}[C]|$, otherwise $H_{k-1}(C)$ is infinite.   

Let $\mathcal{C}(n,k)$ be the set of $k$-dimensional hypertrees over $[n]$. As we discussed  in the previous paragraph, one can also define $\mathcal{C}(n,k)$ as the  set of $k$-dimensional simplicial complexes $C$ over~$[n]$ such that     


\begin{enumerate}[\hspace{0.3cm}(1)]
\item $C$ has a complete $k-1$-skeleton;
\item $C$ has precisely ${{n-1}\choose {k}}$ $k$-faces;
\item the homology group $H_{k-1}(C)$ is finite. 
\end{enumerate} 

From the Cauchy-Binet formula, we see that
\[\det (\hat{I}_{n,k} \hat{I}_{n,k}^T)=\sum_{C\in \mathcal{C}(n,k)} |\det \hat{I}_{n,k}[C]|^2=\sum_{C\in \mathcal{C}(n,k)} |H_{k-1}(C)|^2.\]

Kalai \cite{kalai1983enumeration} was able to evaluate the determinant of $\hat{I}_{n,k} \hat{I}_{n,k}^T$ and obtained the following generalization of Cayley's formula:  
\[\sum_{C\in \mathcal{C}(n,k)} |H_{k-1}(C)|^2=n^{{n-2}\choose{k}}.\]

This formula suggests that if one wants to define a well-behaving probability measure on the set $\mathcal{C}(n,k)$, then the probability of a simplicial complex $C\in \mathcal{C}(n,k)$ should be proportional to $|H_{k-1}(C)|^2$. Thus, we define a probability measure $\nu_{n,k}$ on $\mathcal{C}(n,k)$ by setting
\[\nu_{n,k}(C)=\frac{|H_{k-1}(C)|^2}{n^{{n-2}\choose {k}}}.\]
Random complexes like these were investigated in greater generality by Lyons~\cite{lyons2009random}.  In particular, we know that $\nu_{n,k}$ is a determinantal measure, as we explain in Subsection \ref{secequ}. For $k=1$, the measure $\nu_{n,1}$ is just the uniform measure on the set of spanning trees of the complete graph on the vertex set~$[n]$. 

Answering a question of Linial and Peled \cite{linial2016phase},\footnote{They did not specify which  measure to take on $\mathcal{C}(n,k)$.} we determine the local weak limit $\nu_{n,k}$ as $n$ goes to infinity for any fixed $k$. In case of $k=1$, this recovers the theorem of Grimmett~\cite{grimmett1980random}, which states that the local weak limit of the uniform spanning tree on the sequence of complete graphs is
the Poisson($1$) branching process conditioned to survive forever. This limiting random rooted graph, which is often called the skeleton tree, can be also constructed as follows. Let $(T_0,o_0),(T_1,o_1),\dots$ be an infinite sequence of independent Galton-Watson trees with Poisson($1$) offspring distribution, connect these trees by adding the semi-infinite path $o_0,o_1,\dots$, and consider $o_0$ as the root of this new tree. 

The local weak limit of  uniform spanning trees is also well understood  in several other cases. We have a description of the limit for  convergent sparse graph sequences (Aldous and Lyons~\cite{aldous2007processes}, see~also~\cite{benjamini2001special}), for convergent dense graph sequences (Hladk{\`y}, Nachmias and  Tran \cite{hladky2018local}), and for $d$-regular graphs, where $d \to\infty$ (Nachmias and Peres \cite{nachmias2020local}). In case of complete graphs, by the results of Aldous~\cite{aldous1991continuumI,aldous1991continuumII,aldous1993continuumIII}, we also know various scaling limits of the uniform spanning tree. 

Although the  analogues of these questions for hypertrees were less investigated, for other models of random simplicial complexes, their local weak limits were identified.   One can define the random simplicial complexes $Y_d(n,p)$ which are analogous to Erd\H{o}s-R\'enyi graphs \cite{linial2006homological}. In this case, the local weak limit is known. In fact, local weak convergence plays a crucial role in the paper of Linial and Peled \cite{linial2016phase}, where they determine the asymptotics of the Betti numbers of $Y_d(n,\frac{c}{n})$ 
for every $c>0$. See also \cite{aronshtam2016threshold,aronshtam2015does,aronshtam2013collapsibility}. 

Another model where the local weak limit is known is the so-called $1$-out model investigated  by Linial and Peled \cite{linial2019enumeration}. 
The random $1$-out $k$-complex over $[n]$ is a complex over the vertex set $[n]$ with complete $k-1$ skeleton, in which every $k-1$-face independently selects a uniform random $k$-face that contains it, and these faces are added to the complex. Interestingly, it turns out that for any fixed $k\ge 1$, these random complexes have the same local weak limit as the random $k$-dimensional hypertrees that we study. In other words, random hypertrees with law $\nu_{n,k}$ are locally indistinguishable from random $1$-out $k$-complexes as $n$ tends to infinity. See Subsection \ref{linpeled} for more details.  
 
We define the bipartite graph $L_{n,k}$ as follows. The two color classes of $L_{n,k}$ are ${{[n]}\choose{k+1}}$ and~${{[n]}\choose{k}}$. Given $X\in {{[n]}\choose{k+1}}$ and $Y\in {{[n]}\choose{k}}$, the vertices $X$ and $Y$ are connected in the graph $L_{n,k}$ if and only if $Y\subset X$. For any $V_0\subset V(L_{n,k})= {{[n]}\choose{k+1}}\cup {{[n]}\choose{k}}$, let $L_{n,k}[V_0]$ be the subgraph of $L_{n,k}$ induced by $V_0$. For a simplicial complex $C$, we define  $L_{n,k}[C]$ as  $L_{n,k}[C\cap V(L_{n,k})]$. By pushing forward the measure $\nu_{n,k}$ by the map $C\mapsto L_{n,k}[C]$, we obtain a measure $\nu_{n,k}'$ on the induced subgraphs of $L_{n,k}$. Let $G_{n,k}$ be a random subgraph of $L_{n,k}$ with law $\nu_{n,k}'$.  Note that $G_{n,1}$ is obtained from a uniform random spanning tree of the $n$-vertex complete  graph by subdividing each edge of the tree by a vertex. 

Now, we describe the random rooted tree $(\mathbb{T}_k,o)$ which will turn out to be  the local weak limit of $G_{n,k}$. We call $(\mathbb{T}_k,o)$ the semi-$k$-ary skeleton tree.  This random tree can be generated by  Algorithm \ref{alg1} on the next page. This algorithm also generates a perfect matching $\mathbb{M}_k$ of $\mathbb{T}_k$. At the end $S_i$ will contain the vertices of~$\mathbb{T}_k$ that are at distance $i$ from the root $o$. For a vertex $v$, $c_v$~denotes the number of children of $v$. In Algorithm \ref{alg1}, Poisson($k$) means a random number with  Poisson($k$) distribution, Uniform($1,\dots,c_v$) is a uniform random element of $\{1,\dots,c_v\}$. All the random choices are independent.  See also Subsection \ref{linpeled}.
\begin{algorithm}
\caption{(For generating the semi-$k$-ary skeleton tree $(\mathbb{T}_k,o)$)}\label{alg1}
\begin{algorithmic}
\State {$S_0:=\{o\}$}
\State {Set $\mathbb{T}_k$ to be the tree consisting of the single vertex $o$}
\State {$\mathbb{M}_k:=\emptyset$}
\For {$i:=0,1,\dots$}
\State $S_{i+1}:=\emptyset$
\For {$v\in S_i$} 
\IIf {$i$ is even and $v$ is matched by $\mathbb{M}_k$} $c_v:=$Poisson(k) \EndIIf
\IIf {$i$ is even and $v$ is not matched by $\mathbb{M}_k$} $c_v:=1+$Poisson(k) \EndIIf
\IIf {$i$ is odd} $c_v:=k$ \EndIIf
\State{Create $c_v$ new vertices $u_1^v,u_2^v,\dots, u_{c_v}^v$}
\For{$j:=1,2,\dots,c_v$}
\State{Add $u_j^v$ to $S_{i+1}$}
\State{Add the edge $vu_j^v$ to $\mathbb{T}_k$}
\EndFor
\If{$v$ is not matched by $\mathbb{M}_k$}
\State $j_v:=$Uniform($1,\dots,c_v$)
\State Add $vu_{j_v}^v$ to $\mathbb{M}_k$
\EndIf 
\EndFor
\EndFor
\end{algorithmic}
\end{algorithm}
\begin{remark}
In Algorithm \ref{alg1},  the random choice of $j_v$ is not necessary. By choosing always $j_v=1$, we would get the same distribution on the rooted isomorphism classes of rooted graphs. However, the random choice of $j_v$ will be useful for us to analyze the algorithm.  
\end{remark}

We have explained all the notions in Theorem \ref{thmMain} except the notion of local weak convergence, which we define in Subsection \ref{localweak}. For the readers who are already familiar with local weak convergence, we just point out that when we do the local sampling, we choose a uniform random $k-1$-face as the root, that is, we choose a uniform random element of ${{[n]}\choose{k}}$. \footnote{Another option would be to choose a uniform random vertex of $G_{n,k}$. With some work, one can show that the convergence of these random rooted graphs also follows from our results. The limiting random rooted tree is not the semi-$k$-ary skeleton tree in this case, but it can be easily obtained from it.}

To recover Grimmet's theorem, note that the semi-$1$-ary skeleton tree is obtained from the skeleton tree by subdividing each edge by a vertex. To see this, observe that in $\mathbb{T}_1$, there is a unique semi-infinite path starting from the root with vertices $o_0,o_1,\dots$ such that $o_{2i}o_{2i+1}\in \mathbb{M}_1$ for all $i\ge 0$. Let us delete all the vertices $o_{2i+1}$ for $i\ge 0$. Let $T_{2i}$ be the connected component of $o_{2i}$ in the resulting graph. Then the random trees $(T_{2i},o_{2i})$ are independent, and each of them is obtained from a Galton-Watson tree with Poisson($1$) offspring distribution by subdividing each edge by a vertex. 

A semi-$k$-ary rooted tree is pair $(T,o)$, where $T$ is a locally finite tree, $o\in V(T)$ and all the vertices at odd distance from $o$ have degree $k+1$, or in other words, they have $k$ children. The depth of $(T,o)$ is defined as $\sup_{v\in V(T)} d(o,v)$. Any finite semi-$k$-ary rooted tree has even depth. Let $V_k(T)$ be the set of vertices of $T$ which are at odd distance from $o$, and let $V_{k-1}(T)$ be the set of vertices of $T$ which are at even distance from $o$.  Let $\Aut(T,o)$ be the group of rooted automorphisms of $T$.  
For a semi-$k$-ary rooted tree $(T,o)$ of depth $r$, let $m^*(T,o)$ be the number matchings of $T$ which cover all the vertices of the ball  $B_{r-1}(T,o)$.

Note that $(\mathbb{T}_k,o)$ is a semi-$k$-ary rooted tree. In particular, for any even $r$, the ball $B_r(\mathbb{T}_k,o)$ is a semi-$k$-ary rooted tree. 
 
Theorem \ref{thmMain} will follow immediately once we prove the following statements. 
\begin{lemma}\label{generate}
Let $(T,o)$ be a semi-$k$-ary rooted tree of depth $r<\infty$, then we have
\[ \mathbb{P}(B_r(\mathbb{T}_k,o) \cong (T,o))=\frac{m^*(T,o)\cdot(k!)^{|V_k(T)|}e^{-k|V_{k-1}(B_{r-1}(T,o))|}}{|\Aut(T,o)|}.\]
\end{lemma}

\begin{theorem}\label{thmFO}
Let $(T,o)$ be a semi-$k$-ary rooted tree of depth $r<\infty$, and let $o_n\in {{[n]}\choose{k}}$. Then
\[\lim_{n\to\infty} \mathbb{P}(B_r(G_{n,k},o_n) \cong (T,o))=\frac{m^*(T,o)\cdot(k!)^{|V_k(T)|}e^{-k|V_{k-1}(B_{r-1}(T,o))|}}{|\Aut(T,o)|}.\]
\end{theorem}
\begin{remark}
In Theorem \ref{thmFO}, we can also choose $o_n$ to be a random vertex in ${{[n]}\choose{k}}$ as long as it is independent from $G_{n,k}$.  
\end{remark}

We also have the following quenched version of Theorem \ref{thmFO}.
\begin{theorem}\label{quanched}
Let $(T,o)$ be a semi-$k$-ary rooted tree of depth $r<\infty$. We define the random variables
\[C_n=\left|\left \{\quad Y\in {{[n]}\choose{k}}\quad\Big|\quad B_r(G_{n,k},Y)\cong (T,o)\quad\right\}\right|.\]
We have
\[\lim_{n\to\infty} \frac{C_n}{{{n} \choose {k}}}=\frac{m^*(T,o)\cdot(k!)^{|V_k(T)|}e^{-k|V_{k-1}(B_{r-1}(T,o))|}}{|\Aut(T,o)|}\]
in probability.
\end{theorem}

\textbf{Acknowledgements.} The author is grateful to Mikl\'os Ab\'ert and the anonymous referee for their comments on the manuscript. The author was partially
supported by the ERC Consolidator Grant 648017. This paper was written while the author was affiliated with Central European University, Budapest and Alfr\'ed R\'enyi Institute of Mathematics, Budapest.

\section{Preliminaries}
\subsection{Local weak convergence}\label{localweak}

A rooted graph is a pair $(G,o)$, where $G$ is a locally finite connected graph, and $o$ is a distinguished vertex of $G$ called the root. One can define a measurable structure on the space of isomorphism classes of rooted graphs \cite{aldous2007processes,benjamini2011recurrence}. Thus, it makes sense to speak about random rooted graphs.

Let $G_n=(V_n,U_n,E_n)$ be a sequence of random finite bipartite graphs, that is, for all $n$, $G_n$~is a random bipartite graph with a given proper two coloring $(V_n,U_n)$ of the vertices and an edge set $E_n$.  Let $r$ be a positive integer, and let $(H,o)$ be a finite rooted graph. We define the random variable $D_n=D_{n,(H,o),r}$ on the same probability space as $G_n$ by
\[D_n=\frac{|\{o_n\in V_n\quad|\quad B_r(G_n,o_n)\cong (H,o)\}|}{|V_n|}.\]

Let $o_n$ be a uniform random vertex of $V_n$ (conditioned on $G_n$). Observe that \[\mathbb{E} D_n=\mathbb{P}(B_r(G_n,o_n)\cong (H,o))\text{ and }D_n=\mathbb{P}(B_r(G_n,o_n)\cong (H,o)|G_n).\]

We say that $G_n=(V_n,U_n,E_n)$ converge to a random rooted graph $(G,o)$ in the annealed sense, if for each positive integer $r$ and each  finite rooted graph $(H,o)$, we have
\[\lim_{n\to\infty} \mathbb{E} D_n=\mathbb{P}(B_r(G,o)\cong (H,o)).\]  

We say that $G_n=(V_n,U_n,E_n)$ converge to a random rooted graph $(G,o)$ in the quenched sense, if for each positive integer $r$ and each  finite rooted graph $(H,o)$, we have
\[\lim_{n\to\infty}  D_n=\mathbb{P}(B_r(G,o)\cong (H,o))\]
in probability.


Theorem \ref{thmMain} is about the convergence of the random bipartite graphs \[G_{n,k}=\left({{[n]}\choose k}, V(G_{n,k})\cap {{[n]}\choose k+1},E(G_{n,k})\right).\]  

The quenched notion of convergence is clearly stronger than the annealed one. Theorem \ref{thmMain} is true even in the stronger quenched sense. 
\subsection{Determinantal probability measures}

Let $B$ be an $r\times m$ matrix with linearly independent rows. Assume that the columns of $B$ are indexed with $[m]$. For $X\subset [m]$, let $B[X]$ be the submatrix of $B$ determined by the columns in~$X$. We define the probability measure $\nu_{B}$ on  subsets  of $[m]$ as follows. For $X\subset [m]$, we set
\[\nu_B(\{X\})=
\begin{cases}
\frac{|\det B[X]|^2}{\det(BB^T)}&\text{if }|X|=r,\\
0& \text{otherwise.}
\end{cases} 
\]
A probability measure defined like above is called a determinantal probability measure \cite{lyons2003determinantal,hough2006determinantal}.   

Let $P=(p_{ij})_{i,j\in [m]}$ be the orthogonal projection to the rowspace of $B$.
\begin{lemma}[\cite{lyons2003determinantal}]\label{detprel}
Let $F\subset [m]$, then
\[\nu_B(\{X|F\subset X\subset [m]\})=\det (p_{ij})_{i,j\in F}.\]
\end{lemma}

\begin{lemma}\label{detbecs}
Let $F\subset [m]$, then
\[\nu_B(\{X|F\subset X\subset [m]\})\le \prod_{i\in F} p_{ii}.\]
\end{lemma}
\begin{proof}
This lemma follows from the fact that determinantal probability measures have negative associations \cite[Theorem 6.5]{lyons2003determinantal}. Or alternatively, we can use the fact that for any positive semidefinite matrix $M=(m_{ij})$, we have the  inequality $\det M\le \prod_i m_{ii}$.
\end{proof}
The proof of the next lemma is straightforward.
\begin{lemma}\label{independent}
Let $A_1$ and $A_2$ be two disjoint subsets of $[m]$ such that for all $i\in A_1$ and $j\in A_2$, we have $p_{ij}=0$. Let $X$ be random subset of law $\nu_B$. Then $A_1\cap X$ and $A_2\cap X$ are independent. 
\end{lemma}
\subsection{An equivalent definition of $\nu_{n,k}$}\label{secequ}
We recall a few definitions from the Introduction. Let $I_{n,k}$ be a matrix indexed by ${{[n]}\choose {k}}\times {{[n]}\choose{k+1}}$ defined as follows. Let $X=\{x_0,x_1,\dots,x_k\}\subset [n]$ such that $x_0<x_1<\dots<x_k$. For a $Y\in {{[n]}\choose {k}}$, we set
\[I_{n,k}(Y,X)=\begin{cases}
(-1)^i&\text{if }Y=X\backslash\{x_i\},\\
0&\text{otherwise.} 
\end{cases}
\]   
Let $\hat{I}_{n,k}$ be the submatrix of $I_{n,k}$ determined by the rows indexed by the elements of ${{[n-1]}\choose {k}}$. 
\begin{proposition}\label{IhatI}
The rows of $\hat{I}_{n,k}$ form a basis of the rowspace of $I_{n,k}$.
\end{proposition}
\begin{proof} It is just a rephrasing of \cite[Lemma 1]{kalai1983enumeration}.
\end{proof}

Note that for $C\in \mathcal{C}(n,k)$, $C$ is determined by $C\cap {{[n]}\choose{k+1}}$. Therefore, with a slight abuse of notation, we will regard $\nu_{n,k}$ as a probability measure on the subsets of~${{[n]}\choose {k+1}}$.

\begin{lemma}\label{nunkdet}
We have
\[\nu_{n,k}=\nu_{\hat{I}_{n,k}}.\]
\end{lemma}  
\begin{proof}
See \cite[Proposition 3.1]{lyons2009random}.
\end{proof}

\subsection{The projection matrix corresponding to $\nu_{n,k}$}
We define
\[P_{n,k}=\frac{1}n I_{n,k}^T I_{n,k}.\]

For two sets $X,Y\in {{[n]}\choose {k+1}}$ such that $|X\cap Y|=k$, we introduce the notation \[J(X,Y)=I_{n,k}(X\cap Y,X)\cdot I_{n,k}(X\cap Y,Y).\]

It is straightforward to see that
\[P_{n,k}(X,Y)=\begin{cases}
\frac{k+1}{n}&\text{if $X=Y$,}\\
\frac{1}n J(X,Y)&\text{if $|X\cap Y|=k$,}\\
0&\text{otherwise.}
\end{cases}
\]

\begin{lemma}\label{JJJ}
Let $S$ be a $k-1$ element subset of $[n]$, and let $R=\{r_1,r_2,r_3\}$ be a $3$ element subset of $[n]$ such that $S$ and $R$ are disjoint. Then
\[J(S\cup\{r_1,r_2\},S\cup\{r_2,r_3\})J(S\cup\{r_2,r_3\},S\cup\{r_1,r_3\})J(S\cup\{r_1,r_3\},S\cup\{r_1,r_2\})=-1.\] 
\end{lemma}
\begin{proof}
By symmetry, we may assume that $r_1<r_2<r_3$. For $i=1,2,3$, we define
\[t_i=|\{s\in S|s<r_i\}|.\]
Then we have
\begin{align*}
&I(S\cup\{r_1\},S\cup \{r_1,r_2\})=(-1)^{t_2+1},&I(S\cup\{r_2\},S\cup \{r_1,r_2\})=(-1)^{t_1},\\
&I(S\cup\{r_2\},S\cup \{r_2,r_3\})=(-1)^{t_3+1},&I(S\cup\{r_3\},S\cup \{r_2,r_3\})=(-1)^{t_2},\\
&I(S\cup\{r_1\},S\cup \{r_1,r_3\})=(-1)^{t_3+1},&I(S\cup\{r_3\},S\cup \{r_1,r_3\})=(-1)^{t_1}.
\end{align*} 
Therefore,
\begin{multline*}
J(S\cup\{r_1,r_2\},S\cup\{r_2,r_3\})J(S\cup\{r_2,r_3\},S\cup\{r_1,r_3\})J(S\cup\{r_1,r_3\},S\cup\{r_1,r_2\})\\
=(-1)^{2(t_1+t_2+t_3)+3}=-1.
\end{multline*}
\end{proof}

\begin{lemma}\label{proj}
The matrix $P_{n,k}$ is the orthogonal projection to the rowspace of $\hat{I}_{n,k}$.
\end{lemma}
\begin{proof}
For any $k+1$ element subset $X$ of $[n]$, let $\chi_X\in \mathbb{R}^{{[n]}\choose{k+1}}$ be the characteristic vector of~$X$. The statement follows if we prove the following two claims for every $k+1$ element subset $X$ of~$[n]$.
\begin{claim}
The vector $P_{n,k}\chi_X$ is in the rowspace of $\hat{I}_{n,k}$.  
\end{claim}
\begin{proof}
Since $P_{n,k}=\frac{1}n I_{n,k}^T I_{n,k}$, we see that $P_{n,k}\chi_X$ is in the rowspace of $I_{n,k}$. By Proposition \ref{IhatI}, the rowspace of $\hat{I}_{n,k}$ is the same as the rowspace of $I_{n,k}$. 
\end{proof}
\begin{claim}\label{cl2}
The vector $\chi_X-P_{n,k}\chi_X$ is orthogonal to each row of $\hat{I}_{n,k}$.  
\end{claim}
\begin{proof}
Let $w$ be a row of $\hat{I}_{n,k}$ corresponding to a $k$ element subset $Y$ of $[n-1]$. We need to prove that $\langle \chi_X-P_{n,k}\chi_X,w\rangle=0$. We distinguish three cases:
\begin{enumerate}
\item If $Y\subset X$, then
\begin{align*}
\langle \chi_X-&P_{n,k}\chi_X,w\rangle\\&=\left(1-\frac{k+1}n\right)I_{n,k}(Y,X)-\sum_{a\in [n]\backslash X}\frac{J(X,Y\cup\{a\})}{n}I_{n,k}(Y,Y\cup\{a\})\\
&=\frac{I_{n,k}(Y,X)}{n}\left(n-(k+1)-\sum_{a\in [n]\backslash X}\left(I_{n,k}(Y,Y\cup\{a\})\right)^2\right)\\
&=\frac{I_{n,k}(Y,X)}{n}\left(n-(k+1)-|[n]\backslash X|\right)\\
&=0.
\end{align*}
\item If  $|Y\cap X|=k-1$, then let $S=Y\cap X$. Let $r_1$ and $r_2$ be the two elements of $X\backslash S$. We have
\begin{multline*}\langle \chi_X-P_{n,k}\chi_X,w\rangle\\=-\frac{1}{n}\left(J(X,Y\cup \{r_1\})I_{n,k}(Y,Y\cup\{r_1\})+J(X,Y\cup \{r_2\})I_{n,k}(Y,Y\cup\{r_2\})\right).
\end{multline*}

Note that $J(X,Y\cup \{r_i\})I_{n,k}(Y,Y\cup\{r_i\})\in \{-1,+1\}$, thus \[J(X,Y\cup \{r_1\})I_{n,k}(Y,Y\cup\{r_1\})+J(X,Y\cup \{r_2\})I_{n,k}(Y,Y\cup\{r_2\})=0\] if and only if \[J(X,Y\cup \{r_1\})I_{n,k}(Y,Y\cup\{r_1\})J(X,Y\cup \{r_2\})I_{n,k}(Y,Y\cup\{r_2\})=-1.\] Let $r_3$ be the unique element of $Y\backslash S$. Then
\begin{align*}
J&(X,Y\cup \{r_1\})I_{n,k}(Y,Y\cup\{r_1\})J(X,Y\cup \{r_2\})I_{n,k}(Y,Y\cup\{r_2\})\\&=J(X,Y\cup \{r_1\})J(X,Y\cup \{r_2\})J(Y\cup\{r_1\},Y\cup\{r_2\})\\
&=J(S\cup\{r_1,r_2\},S\cup \{r_1,r_3\})J(S\cup\{r_1,r_2\},S\cup \{r_2,r_3\})J(S\cup\{r_1,r_3\},S\cup\{r_2,r_3\})\\&=-1,
\end{align*}
where in the last step, we used Lemma \ref{JJJ}.
\item If $|Y\cap X|<k-1$, then $\chi_X-P_{n,k}\chi_X$ and $w$ have disjoint support. Thus, they are clearly orthogonal.
\end{enumerate}

Thus, we proved Claim \ref{cl2}.
\end{proof}
This finishes the proof of Lemma \ref{proj}.
\end{proof}

By combining Lemma \ref{detprel}, Lemma \ref{nunkdet} and Lemma \ref{proj}, we obtain the following lemma.
\begin{lemma}\label{lemmadet}
Let $F\subset {{[n]}\choose{k+1}}$, then
\[\mathbb{P}(F\subset V(G_{n,k}))=\det (P_{n,k}(A,B))_{A,B\in F}.\]
\end{lemma}
Using Lemma \ref{detbecs} and the fact that all the diagonal entries of $P_{n,k}$ are equal to~$\frac{k+1}n$, we get the following lemma.
\begin{lemma}\label{detbecs2}
Let $F\subset {{[n]}\choose{k+1}}$, then
\[\mathbb{P}(F\subset V(G_{n,k}))\le \left(\frac{k+1}{n}\right)^{|F|}.\]
\end{lemma}

%
\subsection{Another description of the semi-$k$-ary skeleton tree}\label{linpeled}
Linial and Peled~\cite{linial2019enumeration} determined the local weak limit of random $1$-out $k$ complexes which is actually equal to the semi-k-ary skeleton tree.  However,  it might not be obvious at first sight that the algorithm described by them indeed generates the semi-k-ary skeleton tree. The aim of this subsection is to make this fact clearer. 

Let us consider the random triple $(\mathbb{T}_k,o,\mathbb{M}_k)$ generated by Algorithm \ref{alg1}. We say that a vertex $v$ of $\mathbb{T}_k$ has type $A$, if either
\begin{itemize}
\item $v$ is at even distance from the root, and $v$ is matched by $\mathbb{M}_k$ to one of its children, or
\item $v$ is at odd distance from the root, and $v$ is matched by $\mathbb{M}_k$ to its parent.
\end{itemize}  
Any other vertex will have type $B$, that is, a vertex $v$ has type $B$, if either
\begin{itemize}
\item $v$ is at even distance from the root, and $v$ is matched by $\mathbb{M}_k$ to its parent, or
\item $v$ is at odd distance from the root, and $v$ is matched by $\mathbb{M}_k$ to one of its children.
\end{itemize}  
Note that the root always have type $A$. Moreover,
\begin{itemize}
\item Any type $A$ vertex $v$ at even distance from the root has a unique type $A$ child, namely, the vertex which is matched to $v$ by $\mathbb{M}_k$.
\item For a type $A$ vertex $v$ at odd distance from the root, all the children of $v$ have type $A$. 
\item For a type $B$ vertex $v$ at even distance from the root, all the children of $v$ have type $B$.
\item Any type $B$ vertex $v$ at odd distance from the root has a unique type $B$ child, namely, the vertex which is matched to $v$ by $\mathbb{M}_k$. 
\end{itemize}
The observations above also imply that in the perfect matching $\mathbb{M}_k$, each vertex is matched to a vertex with the same type. Using these facts, it is easy to see that the semi-$k$-ary tree can be also generated by Algorithm \ref{alg2}. This algorithm is the same as the algorithm that was used by Linial and Peled \cite{linial2019enumeration} to describe the local weak limit of the random $1$-out $k$-complexes. Therefore, random $1$-out $k$-complexes have the same local weak limit as random $k$-dimensional hypertrees.

\begin{algorithm}
\caption{(For generating the semi-$k$-ary skeleton tree $(\mathbb{T}_k,o)$) with types}\label{alg2}
\begin{algorithmic}
\State {$S_0:=\{o\}$}
\State {Set $\mathbb{T}_k$ to be the tree consisting of the single vertex $o$}
\State {$\mathbb{M}_k:=\emptyset$}
\State {$type(o):=A$}
\For {$i:=0,1,\dots$}
\State $S_{i+1}:=\emptyset$
\For {$v\in S_i$} 
\IIf {$i$ is even and $type(v)=B$} $c_v:=$Poisson(k) \EndIIf
\IIf {$i$ is even and $type(v)=A$} $c_v:=1+$Poisson(k) \EndIIf
\IIf {$i$ is odd} $c_v:=k$ \EndIIf
\State{Create $c_v$ new vertices $u_1^v,u_2^v,\dots, u_{c_v}^v$}
\For{$j:=1,2,\dots,c_v$}
\State{Add $u_j^v$ to $S_{i+1}$}
\State{Add the edge $vu_j^v$ to $\mathbb{T}_k$}
\EndFor
\If{$i$ is even and $type(v)=A$}
\State $j_v:=$Uniform($1,\dots,c_v$)
\State Add $vu_{j_v}^v$ to $\mathbb{M}_k$
\State $type(u_{j_v}^v):=A$
\FFor {$j\in \{1,\dots,c_v\}\backslash\{j_v\}$} $type(u_j^v):=B$ \EndFFor
\EndIf 
\If{$i$ is even and $type(v)=B$} \FFor {$j:=1,2\dots,c_v$}  $type(u_j^v):=B$\EndFFor \EndIf 
\If{$i$ is odd and $type(v)=A$} \FFor {$j:=1,2\dots,c_v$}  $type(u_j^v):=A$\EndFFor \EndIf 
\If{$i$ is odd and $type(v)=B$}
\State $j_v:=$Uniform($1,\dots,c_v$)
\State Add $vu_{j_v}^v$ to $\mathbb{M}_k$
\State $type(u_{j_v}^v):=B$
\FFor {$j\in \{1,\dots,c_v\}\backslash\{j_v\}$} $type(u_j^v):=A$ \EndFFor
\EndIf

\EndFor
\EndFor
\end{algorithmic}
\end{algorithm}
 
\section{The local weak limit of $G_{n,k}$}

\subsection{Formulas for fixed $n$}
Given any subgraph $H$ of $L_{n,k}$, we define
\begin{align*}
V_k(H)&=V(H)\cap {{[n]}\choose{k+1}}\text{ and}\\
V_{k-1}(H)&=V(H)\cap {{[n]}\choose{k}}.
\end{align*}

We say that $T$ is a proper subtree of $L_{n,k}$ if $T$ is an induced subtree of $L_{n,k}$ such that if $X\in V_k(T)$, then all the $k+1$ neighbors of $X$ in the graph $L_{n,k}$ are also in $V(T)$.

For a proper subtree $T$ of  $L_{n,k}$ a matching $M$ of $T$ is called complete if it covers all the vertices in $V_k(T)$. Let $m(T)$ be the number of complete matchings in $T$.

\begin{proposition}\label{uniqm}
Let $T$ be a proper subtree of $L_{n,k}$, and let $U\subset  V_{k-1}(T)$. Then there is at most one complete matching of $T$ that covers precisely the vertices in~$V_k(T)\cup U$.
\end{proposition} 
\begin{proof}
We prove by contradiction. Assume that there are two distinct matchings of $T$ with the given property. Then the symmetric difference of these two matchings is non-empty and it is a vertex disjoint union of cycles. But since $T$ is a tree, $T$ has no cycles, which is a contradiction. 
\end{proof}

\begin{lemma}\label{probT}
For any proper subtree $T$ of $L_{n,k}$, we have
\[\mathbb{P}(T\subset G_{n,k})=\frac{m(T)}{n^{|V_k(T)|}}.\]
\end{lemma}
\begin{proof}
Let $Q$ be the submatrix of $P_{n,k}$ determined by the rows and columns corresponding to $V_k(T)$. Let $A=I_{n,k}[V_k(T)]$. Then $Q=\frac{1}{n}A^T A$. Combining Lemma~\ref{lemmadet} with the Cauchy-Binet formula, we obtain that
 \[\mathbb{P}(T\subset G_{n,k})=\det Q=n^{-|V_k(T)|}\det(A^T A)=n^{-|V_k(T)|}\sum_{U} |\det A^T[U]|^2,\]
 where the summation is over the subsets $U$ of ${{[n]}\choose{k}}$ such that $U$ has the same cardinality as~$V_k(T)$. 
 
 Observe that in the Leibniz formula for the determinant of $A^T[U]$,  the non-zero terms correspond to the perfect matchings in $L_{n,k}[V_k(T)\cup U]$. Note that if $L_{n,k}[V_k(T)\cup U]$ has a perfect matching, then $L_{n,k}[V_k(T)\cup U]$ is alway an induced subgraph of $T$, because $T$ is a proper subtree. In this case, by Proposition~\ref{uniqm}, $L_{n,k}[V_k(T)\cup U]$ has a unique perfect matching. Therefore, $|\det A^T[U]|=1$. It also follows from the observations above that the number of sets $U$ such that $|\det A^T[U]|=1$ is equal to the number of complete matchings in $T$. Therefore,
 \[\mathbb{P}(T\subset G_{n,k})=n^{-|V_k(T)|}\sum_{U} |\det A^T[U]|^2=\frac{m(T)}{n^{|V_k(T)|}}.\] 
\end{proof}

\begin{lemma}\label{probTT}
Let $T$ and $T'$ be two proper subtrees of $L_{n,k}$. Moreover, assume that $T$ is a subtree of $T'$ such that for any $X\in V_k(T')\backslash V_k(T)$, there is a neighbor of $X$ which is in $V(T)$.\footnote{This neighbor is necessarily unique.} Let $M$ be a complete matching of $T$. We define $m(T,T',M)$ as the number of complete matchings $M'$ of $T'$ such that $M\subset M'$. Let $U=U_{T,M}$ be the set of vertices in $V_{k-1}(T)$ which are uncovered by~$M$. For $Y\in V_{k-1}(T)$, let $\Delta(Y)=\Delta_{T,T'}(Y)=\deg_{T'}(Y)-\deg_T(Y)$. Then
\[m(T,T',M)=\left(\prod_{Y\in U}\left(\Delta(Y)k^{\Delta(Y)-1}+k^{\Delta(Y)}\right)\right)\left(\prod_{Y\in V_{k-1}(T)\backslash U}k^{\Delta(Y)}\right).\]
Moreover,
\begin{multline*}
\mathbb{P}(T'\subset G_{n,k})=\\n^{-|V_k(T')|}\sum_{M} \left(\prod_{Y\in U_{T,M}}\left(\Delta(Y)k^{\Delta(Y)-1}+k^{\Delta(Y)}\right)\right)\left(\prod_{Y\in V_{k-1}(T)\backslash U_{T,M}}k^{\Delta(Y)}\right),
\end{multline*}
where the summation is over the complete matchings of $T$. 
\end{lemma}
\begin{proof}
First, we count the the number of complete matchings $M'$ of $T'$ such that $M\subset M'$. 
For any vertex $v\in V(L_{n,k})$, let 
\[D_v=\{u\in V(T')\backslash V(T)| \text{ $u$ is a neighbor of $v$ in $L_{n,k}$}\}.\]
It follows from the assumptions of the lemma that $V_k(T')\backslash V_k(T)=\cup_{Y\in V_{k-1}(T)} D_Y$, where this is a disjoint union. Take any $Y\in  V_{k-1}(T)\backslash U_{T,M}$. In this case, any $X\in D_Y$ must be matched by $M'$ to a vertex in the $D_X$. Note that $|D_X|=k$. Thus, there are $k^{|D_Y|}=k^{\Delta(Y)}$ ways to match the vertices of $D_Y$. If $Y\in U_{T,M}$, then there is the additional option to match one vertex $X\in D_Y$ to $Y$, after this, each vertex  $X'\in D_Y\backslash \{X\}$ must be matched to a vertex in the $k$-element set~$D_{X'}$. This provides us $|D_Y| k^{|D_Y|-1}=\Delta(Y) k^{\Delta(Y)-1}$ additional ways to match the vertices of $D_Y$. Therefore, if  $Y\in U_{T,M}$, then there are  $\Delta(Y)k^{\Delta(Y)-1}+k^{\Delta(Y)}$ possibilities altogether to match the vertices of $D_Y$. As  these choices can be made independently for all $Y\in V_{k-1}(T)$, the first formula follows.

The second formula is a direct consequence of the first one and Lemma \ref{probT}.
\end{proof}

The pair $(T,o)$ is called a rooted proper subtree of $L_{n,k}$ if $T$ is proper subtree of $L_{n,k}$ and $o\in V_{k-1}(T)$. A rooted proper subtree $(T,o)$ has depth $r$, if all the vertices of $T$ is at most distance $r$ from $o$, and there is at least one vertex of $T$ which is at distance $r$ from $o$. By distance, we refer to the graph distance of $T$. The depth of rooted proper subtree is always even. Let $V_{n,k}^{(T,o)}$ be the  set of $k+1$ element subsets $X$ of $[n]$  such that $X$ has a neighbor  in the graph $L_{n,k}$ which is  in $V_{k-1}(B_{r-2}(T,o))$, where $r$ is the depth of $(T,o)$.  

The proof of the following proposition is straightforward. 
\begin{proposition}
Let $(T,o)$ be a proper rooted subtree of $L_{n,k}$ of depth $r$, then
\[\mathbb{P}(B_r(G_{n,k},o)=T)=\mathbb{P}(V_{n,k}^{(T,o)}\cap V_k(G_{n,k})=V_k(T)).\]
\end{proposition}

Using the inclusion-exclusion formula, this leads to the formula 
\[\mathbb{P}(B_r(G_{n,k},o)=T)=\sum_{W\subset V_{n,k}^{(T,o)}\backslash V_k(T) }(-1)^{|W|} \mathbb{P}(W\cup V_k(T)\subset V_k(G_{n,k})).\]

For $W\subset V_{n,k}^{(T,o)}\backslash V_k(T)$, let $G_W=G_{W,T}$ be $L_{n,k}[(W\cup V_k(T))\cup \Gamma(W\cup V_k(T))]$, where $\Gamma(W\cup V_k(T))$ is the set of neighbors of $W\cup V_k(T)$ in $L_{n,k}$. With this notation, we have
\begin{equation}\label{incexc}
\mathbb{P}(B_r(G_{n,k},o)=T)=\sum_{W\subset V_{n,k}^{(T,o)}\backslash V_k(T) }(-1)^{|W|} \mathbb{P}(G_W\subset G_{n,k}).
\end{equation}
If $G_W$ is a tree, then the conditions of Lemma \ref{probTT} are satisfied with $T'=G_W$, thus, we have the formula
\begin{multline*}
\mathbb{P}(G_W\subset G_{n,k})=
n^{-|V_k(G_W)|}\\\times\sum_{M} \left(\prod_{Y\in U_{T,M}}(\Delta_{T,G_W}(Y)+k)k^{\Delta_{T,G_W}(Y)-1}\right)\left(\prod_{Y\in V_{k-1}(T)\backslash U_{T,M}}k^{\Delta_{T,G_W}(Y)}\right),
\end{multline*}
where the summation is over the complete matchings of $T$.

If $G_W$ is a tree, then $\Delta_{T,G_W}(Y)=0$ for any $Y\in V_{k-1}(B_r(T,o))\backslash V_{k-1}(B_{r-2}(T,o))$. 

Thus, introducing the notations \[U'_{T,M}=U_{T,M}\cap V_{k-1}(B_{r-2}(T,o))\text{ and }U^c_{T,M}=V_{k-1}(B_{r-2}(T,o))\backslash U'_{T,M},\] we have
\begin{multline*}
\mathbb{P}(G_W\subset G_{n,k})\\=
n^{-|V_k(G_W)|}\sum_{M} \left(\prod_{Y\in U'_{T,M}}(\Delta_{T,G_W}(Y)+k)k^{\Delta_{T,G_W}(Y)-1}\right)\left(\prod_{Y\in U^c_{T,M}}k^{\Delta_{T,G_W}(Y)}\right).
\end{multline*}
Also note that $|V_k(G_W)|=|W|+|V_k(T)|$. Therefore, assuming that $G_W$ is a tree, we have
\begin{multline}\label{treeincexc}
n^{|V_k(T)|}\mathbb{P}(G_W\subset G_{n,k})\\=
n^{-|W|}\sum_{M} \left(\prod_{Y\in U'_{T,M}}(\Delta_{T,G_W}(Y)+k)k^{\Delta_{T,G_W}(Y)-1}\right)\left(\prod_{Y\in U^c_{T,M}}k^{\Delta_{T,G_W}(Y)}\right).
\end{multline}
\subsection{An asymptotic formula}
For any $o_n\in V_{k-1}(L_{n,k})$ and a semi-$k$-ary rooted tree $(T,o)$, we define
\begin{multline*}\mathcal{I}_{n,o_n}(T,o)=\\\{T_n|\text{ $T_n$ is an induced subgraph of $L_{n,k}$, $o_n\in V(T_n)$ and $(T,o)\cong (T_n,o_n)$ } \}.
\end{multline*}
We also define
\[\mathcal{I}_{n}(T,o)=\{(T_n,o_n)|o_n\in V_{k-1}(L_{n,k}), T_n\in \mathcal{I}_{n,o_n}(T,o) \}.\] 

For a  semi-$k$-ary rooted tree $(T,o)$ of depth $r$, we define $\inte(T,o)=|V_{k-1}(B_{r-2}(T,o))|$. 
\begin{lemma}\label{nehezlemma}
Consider a  semi-$k$-ary rooted tree $(T,o)$. For every  $n$, let us choose $(T_n,o_n)\in \mathcal{I}_{n}(T,o)$. For any choice of $(T_n,o_n)$, we have
\[\lim_{n\to\infty}n^{|V_k(T)|}\mathbb{P}(B_r(G_{n,k},o_n)=T_n)=m^*(T,o)e^{-k\cdot \inte(T,o)}.\]
In other words,
\[\lim_{n\to\infty} \max_{(T_n,o_n)\in \mathcal{I}_{n}(T,o)} \left|n^{|V_k(T)|}\mathbb{P}(B_r(G_{n,k},o_n)=T_n)-m^*(T,o)e^{-k\cdot \inte(T,o)}\right|=0.\]
\end{lemma}  
\begin{proof}

For simplicity of notation, we drop the index $n$ from the notation $(T_n,o_n)$.

By Equation~\eqref{incexc}, we have
\[n^{|V_k(T)|}\mathbb{P}(B_r(G_{n,k},o)=T)=\sum_{W\subset V_{n,k}^{(T,o)}\backslash V_k(T) }(-1)^{|W|} n^{|V_k(T)|}\mathbb{P}(G_W\subset G_{n,k}).\]

The next lemma states that in the formula above, asymptotically it is enough to  consider only the terms where $G_W$ is a tree.

\begin{lemma} We have\footnote{To be more precise, we only prove that if one of the limits exists, then the other limit exists too, and these limits are equal. We will not emphasize this distinction anymore.}
\[\lim_{n\to \infty} n^{|V_k(T)|}\mathbb{P}(B_r(G_{n,k},o)=T)=\lim_{n\to\infty}\sum_{\substack{W\subset V_{n,k}^{(T,o)}\backslash V_k(T)\\G_W\text{ is a tree}} }(-1)^{|W|} n^{|V_k(T)|}\mathbb{P}(G_W\subset G_{n,k}).
\]
\end{lemma} 
\begin{proof}
Let 
\[S_{n,i}=\sum_{\substack{W\subset V_{n,k}^{(T,o)}\backslash V_k(T)\\G_W\text{ is not a tree}\\|W|=i} } n^{|V_k(T)|}\mathbb{P}(G_W\subset G_{n,k}).\]
It is enough to prove that 
\[\lim_{n\to \infty}\sum_{i=0}^{\infty} S_{n,i}=0.\]
Note that $|V_{n,k}^{(T,o)}|\le \inte(T,o)n$. Moreover, for $W\subset V_{n,k}^{(T,o)}\backslash V_k(T)$ such that $|W|=i$, we have
\[n^{|V_k(T)|}\mathbb{P}(G_W\subset G_{n,k})\le n^{|V_k(T)|} \left(\frac{k+1}{n}\right)^{|V_k(T)|+i}=(k+1)^{|V_k(T)|}\left(\frac{k+1}{n}\right)^i,\]
by Lemma \ref{detbecs2}. 
Thus,

\begin{align*}
S_{n,i}&\le {{\inte(T,o)n}\choose{i}} (k+1)^{|V_k(T)|}\left(\frac{k+1}{n}\right)^i\\&\le \frac{(\inte(T,o)n)^i}{i!}(k+1)^{|V_k(T)|}\left(\frac{k+1}{n}\right)^i\\&=(k+1)^{|V_k(T)|}\frac{((k+1)\inte(T,o))^i}{i!}.
\end{align*}
Also, note that
\[\sum_{i=0}^{\infty}  (k+1)^{|V_k(T)|}\frac{((k+1)\inte(T,o))^i}{i!}=(k+1)^{|V_k(T)|}\exp((k+1)\inte(T,o))<\infty.\]
Therefore, by the dominated convergence theorem, it is enough to prove that for any $i$, we have
\[\lim_{n\to \infty} S_{n,i}=0.\]

\begin{claim}\label{cldb}
Let \[D_n=\{W\subset V_{n,k}^{(T,o)}\backslash V_k(T)\quad|\quad|W|=i\text{ and $G_W$ is not a tree}\}.\]
Then
\[|D_n|\le Cn^{i-1}\]
for some $C$ which only depends on $(T,o)$ and $i$, but not on $n$.
\end{claim}
\begin{proof}
We give an injection from $D_n$ to \[R_n=\left(V_{n,k}^{(T,o)}\cup \left(([|V_k(T)|+i]\times [k+1])^2\right)\right)^i\backslash \left(V_{n,k}^{(T,o)}\right)^i.\]  

First, we list the elements of $V_k(T)$ in lexicographical order as $X_1,X_2,\dots,X_m$, where $m=|V_k(T)|$. 

Take any $W\in D_n$. Let $X_{m+1},X_{m+2},\dots,X_{m+i}$ be the list of elements of $W$ in lexicographic order. Let $W_j=\{X_{m+1},\dots,X_{m+j}\}$. We map $W$ to $\varphi(W)=(L_1,L_2,\dots,L_i)\in R_n$ as follows. For every $j=1,2,\dots,i$, we define $L_j$ as follows.
\begin{enumerate}[(i)]
\item If $X_{m+j}$ has at most one neighbor in the graph $L_{n,k}$ which is in $V_{k-1}(G_{W_{j-1}})$, then we set $L_j=X_{m+j}$.
\item If $X_{m+j}$ has at least two neighbors in the graph $L_{n,k}$ which are in $V_{k-1}(G_{W_{j-1}})$, then we proceed as follows. Let $Y_1$ and $Y_2$ be the first two such neighbors in lexicographic order. Note that $X_{m+j}=Y_1\cup Y_2$. Then there are $h_1,h_2<m+j$ such that $Y_1$ is a neighbor of $X_{h_1}$ and $Y_2$ is a neighbor of $X_{h_2}$. (If there are more indices like that just pick the smallest ones.) Now let us choose $\ell_1\in [k+1]$ such that $Y_1$ is the $\ell_1$th neighbor of $X_{h_1}$ in  lexicographic order. We define $\ell_2$ the analogous way.  We set $L_j=((h_1,\ell_1),(h_2,\ell_2))$.
\end{enumerate}   
It is straightforward to verify that $\varphi$ is indeed an injection. Moreover, since $G_W$ is not a tree, the case (ii) above occurs at least once, thus, indeed $\varphi(W)\in R_n$.

Since $|V_{n,k}^{(T,o)}|\le \inte(T,o)n$, we have \[|R_n|\le (\inte(T,o)n+(|V_k(T)|+i)^2(k+1)^2)^i-(\inte(T,o)n)^i\le Cn^{i-1}\] for a sufficiently large $C$. 
\end{proof}
Using Lemma \ref{detbecs2}, 
we obtain that
\begin{align*}S_{n,i}&\le |D_n| n^{|V_k(T)|} \left(\frac{(k+1)}{n}\right)^{|V_k(T)|+i}\\&\le Cn^{i-1} n^{|V_k(T)|} \left(\frac{(k+1)}{n}\right)^{|V_k(T)|+i}\\&=\frac{C(k+1)^{|V_k(T)|+i}}{n}.
\end{align*}
Thus,
\[\lim_{n\to\infty} S_{n,i}\le \lim_{n\to\infty} \frac{C(k+1)^{|V_k(T)|+i}}{n}=0.\]
\end{proof}

From Equation \eqref{treeincexc}, we have
\begin{multline*}
\sum_{\substack{W\subset V_{n,k}^{(T,o)}\backslash V_k(T) \\G_W\text{ is a tree}}}(-1)^{|W|} n^{|V_k(T)|}\mathbb{P}(G_W\subset G_{n,k})=\sum_{M}\sum_{\substack{W\subset V_{n,k}^{(T,o)}\backslash V_k(T) \\G_W\text{ is a tree}}}\\(-1)^{|W|} n^{-|W|} \left(\prod_{Y\in U'_{T,M}}(\Delta_{T,G_W}(Y)+k)k^{\Delta_{T,G_W}(Y)-1}\right)\left(\prod_{Y\in U^c_{T,M}}k^{\Delta_{T,G_W}(Y)}\right).
\end{multline*}
For any $n$ and $t\in \mathbb{N}_0^{V_{k-1}(B_{r-2}(T,o))}$, let $N(n,t)$ be the number of $W\subset V_{n,k}^{(T,o)} \backslash V_k(T)$ such that $G_W$ is a tree, and $\Delta_{T,G_W}(Y)=t_Y$ for any $Y\in V_{k-1}(B_{r-2}(T,o))$. Then
\begin{multline*}
\sum_{\substack{W\subset V_{n,k}^{(T,o)}\backslash V_k(T) \\G_W\text{ is a tree}}}(-1)^{|W|} n^{|V_k(T)|}\mathbb{P}(G_W\subset G_{n,k})\\=\sum_{M}\sum_{t}N(n,t)(-1)^{\sum_Y t_Y} n^{-\sum_Y t_Y} \left(\prod_{Y\in U'_{T,M}}(t_Y+k)k^{t_Y-1}\right)\left(\prod_{Y\in U^c_{T,M}}k^{t_Y}\right),
\end{multline*}
where $\sum_t$ stands for $\sum_{t\in \mathbb{N}_0^{V_{k-1}(B_{r-2}(T,o))}}$, and $\sum_Y$ stands for $\sum_{Y\in V_{k-1}(B_{r-2}(T,o))}$.

The next lemma states that $N(n,t)$ can be approximated by $\prod_Y {{n}\choose{t_Y}}$.
\begin{lemma}\label{lemma24}
We have
\begin{align*}
\lim_{n\to\infty}&\sum_{M}\sum_{t}N(n,t)(-1)^{\sum_Y t_Y} n^{-\sum_Y t_Y} \left(\prod_{Y\in U'_{T,M}}(t_Y+k)k^{t_Y-1}\right)\left(\prod_{Y\in U^c_{T,M}}k^{t_Y}\right)\\&=\lim_{n\to \infty} \sum_{M}\sum_{t}\left(\prod_Y {{n}\choose{t_Y}}\right)\\&\qquad\qquad(-1)^{\sum_Y t_Y} n^{-\sum_Y t_Y} \left(\prod_{Y\in U'_{T,M}}(t_Y+k)k^{t_Y-1}\right)\left(\prod_{Y\in U^c_{T,M}}k^{t_Y}\right).
\end{align*}
\end{lemma} 
\begin{proof}
The next claim provides us an estimate on $N(n,t)$.
\begin{claim}\label{claim22}
We have
\[\frac{\prod_{j=0}^{\sum_{Y}t_Y-1}(n-(k+1)|V_k(T)|-j)}{\prod_{Y}t_Y!}\le N(n,t)\le \prod_Y {{n}\choose{t_Y}}.\]
\end{claim}
\begin{proof}
For each vertex $Y\in V_{k-1}(B_{r-2}(T,o))$, we have $n-k-\deg_T(Y)\le n$ neighbors of $Y$ in $L_{n,k}$ which are not in $V_k(T)$. Thus, the upper bound follows easily. 

Let $B=\cup_{X\in V_k(T)} X$. For each $Y\in V_{k-1}(B_{r-2}(T,o))$, we assign a $t_Y$ element subset $E_Y$ of $[n]\backslash B$ such that these subsets are pairwise disjoint. \break Since $|[n]\backslash B|\ge n-(k+1)|V_k(T)|$, there are at least \[\frac{\prod_{j=0}^{\sum_{Y}t_Y-1}(n-(k+1)|V_k(T)|-j)}{\prod_{Y}t_Y!}\]
such assignments. 

Given such an assignment, let
\[W=\{Y\cup\{e\}|Y\in V_{k-1}(B_{r-2}(T,o)),\quad e\in E_Y\}.\]
Clearly, that map $(E_Y)\mapsto W$ is injective. Moreover, $W\subset V_{n,k}^{(T,o)} \backslash V_k(T)$, and $\Delta_{T,G_W}(Y)=t_Y$ for any $Y=V_{k-1}(B_{r-2}(T,o))$. We only need to prove that $G_W$ is a tree.

Consider $Y\in V_{k-1}(B_{r-2}(T,o))$ and $e\in E_Y$. Let \[\Gamma_{(Y,e)}=\{Y'|Y'\subset Y\cup\{e\}, |Y'|=k, Y'\neq Y\}.\] Note that all the elements of $\Gamma_{(Y,e)}$ contain $e$. Thus, these sets are pairwise disjoint and they are also disjoint from $V_{k-1}(T)$. Moreover, since any $Y'\in \Gamma_{(Y,e)}$ contains~$e$, the only neighbor of $Y'$ in $G_W$ is $Y\cup\{e\}$. From these observations it follows that $G_W$ is a tree.  
\end{proof}

We define 
\[\mathcal{T}_i=\left\{t\in \mathbb{N}_0^{V_{k-1}(B_{r-2}(T,o))}|\sum_Y t_Y=i \right\},\] 
and 
\[q_{n,i}=\min_{t\in \mathcal{T}_i} \frac{N(n,t)}{\prod_Y {{n}\choose{t_Y}}}.\] 

From Claim \ref{claim22}, it follows that
\[1\ge q_{n,i}\ge  \min_{t\in \mathcal{T}_i} \frac{\frac{\prod_{j=0}^{i-1}(n-(k+1)|V_k(T)|-j)}{\prod_{Y}t_Y!}}{\prod_Y {{n}\choose{t_Y}}}\ge \frac{(n-(k+1)|V_k(T)|-i)^i}{n^i}.\]
Thus,
\[\lim_{n\to\infty} q_{n,i}=1.\]

Let us define
\[S_{n,i}=\sum_{M}\sum_{t\in \mathcal{T}_i}N(n,t) n^{-i} \left(\prod_{Y\in U'_{T,M}}(t_Y+k)k^{t_Y-1}\right)\left(\prod_{Y\in U^c_{T,M}}k^{t_Y}\right),\]
and
\[\bar{S}_{n,i}=\sum_{M}\sum_{t\in \mathcal{T}_i}n^{-i} \left(\prod_Y {{n}\choose{t_Y}}\right)  \left(\prod_{Y\in U'_{T,M}}(t_Y+k)k^{t_Y-1}\right)\left(\prod_{Y\in U^c_{T,M}}k^{t_Y}\right).\]
Observe that
\[q_{n,i} \bar{S}_{n,i}\le S_{n,i}\le \bar{S}_{n,i}.\]
Thus,\footnote{If the limits do not exist, we can just take a suitable subsequence, because $S_{n,i}$ are uniformly bounded by Equation \eqref{unifbound}. We omit the details.}
\[\lim_{n\to\infty} \bar{S}_{n,i}=\lim_{n\to\infty }S_{n,i}.\]
Note that $|\mathcal{T}_i|\le (i+1)^{|{V_{k-1}(B_{r-2}(T,o))}|}$. Moreover, for any $t\in \mathcal{T}_i$, there is  an $Y\in {V_{k-1}(B_{r-2}(T,o))}$, such that $t_Y\ge  i/|{V_{k-1}(B_{r-2}(T,o))}|$. Consequently, we have 
\[\prod_Y {{n}\choose{t_Y}}\le \frac{n^i}{\lfloor i/|{V_{k-1}(B_{r-2}(T,o))}|\rfloor!}.\]
Therefore,
\begin{align}\label{unifbound}\bar{S}_{n,i}&\le m(T) (i+1)^{|{V_{k-1}(B_{r-2}(T,o))}|} n^{-i} \frac{n^i}{\lfloor i/|{V_{k-1}(B_{r-2}(T,o))}|\rfloor!}(k+i)^{|{V_{k-1}(B_{r-2}(T,o))}|} k^i\\\nonumber &= \frac{m(T) (i+1)^{|{V_{k-1}(B_{r-2}(T,o))}|}  (k+i)^{|{V_{k-1}(B_{r-2}(T,o))}|} k^i}{\lfloor i/|{V_{k-1}(B_{r-2}(T,o))}|\rfloor!}.
\end{align}
Since 
\[\sum_{i=0}^{\infty }\frac{m(T) (i+1)^{|{V_{k-1}(B_{r-2}(T,o))}|}  (k+i)^{|{V_{k-1}(B_{r-2}(T,o))}|} k^i}{\lfloor i/|{V_{k-1}(B_{r-2}(T,o))}|\rfloor!}<\infty,\]
we can use the dominated convergence theorem to obtain that
\[\lim_{n\to\infty}\sum_{i=0}^{\infty} (-1)^i S_{n,i}=\lim_{n\to\infty}\sum_{i=0}^{\infty} (-1)^i \bar{S}_{n,i}.\] 
This is exactly the statement of Lemma \ref{lemma24}.
\end{proof}  
For any $n$ and $M$, we have
\begin{align*}
&\sum_{t}\left(\prod_Y {{n}\choose{t_Y}}\right)(-1)^{\sum_Y t_Y} n^{-\sum_Y t_Y} \left(\prod_{Y\in U'_{T,M}}(t_Y+k)k^{t_Y-1}\right)\left(\prod_{Y\in U^c_{T,M}}k^{t_Y}\right)\\
&=\sum_{t}\left(\prod_{Y\in U'_{T,M}}{{n}\choose{t_Y}}\left(\left(-\frac{k}{n}\right)^{t_Y}-\frac{t_Y}{n}\left(-\frac{k}{n}\right)^{t_Y-1}\right)\right)\left(\prod_{Y\in U^c_{T,M}}{{n}\choose{t_Y}}\left(-\frac{k}{n}\right)^{t_Y}\right)\\
&=\left(\prod_{Y\in U'_{T,M}}\sum_{t_Y=0}^n{{n}\choose{t_Y}}\left(\left(-\frac{k}{n}\right)^{t_Y}-\frac{t_Y}{n}\left(-\frac{k}{n}\right)^{t_Y-1}\right)\right)\cdot\\&\qquad\qquad\qquad\qquad\qquad\qquad\qquad\qquad\qquad\qquad\quad\left(\prod_{Y\in U^c_{T,M}}\sum_{t_Y=0}^n{{n}\choose{t_Y}}\left(-\frac{k}{n}\right)^{t_Y}\right)\\
&=\left(\sum_{i=0}^n{{n}\choose{i}}\left(\left(-\frac{k}{n}\right)^{i}-\frac{i}{n}\left(-\frac{k}{n}\right)^{i-1}\right)\right)^{|U'_{T,M}|}\left(\sum_{i=0}^n{{n}\choose{i}}\left(-\frac{k}{n}\right)^{i}\right)^{|U^c_{T,M}|}.
\end{align*}

Now observe that 
\[\lim_{n\to\infty} \sum_{i=0}^n {{n}\choose{i}} \left(-\frac{k}{n}\right)^i=\lim_{n\to\infty} \left(1-\frac{k}n\right)^n=e^{-k},\]
and 
\begin{align*}
\lim_{n\to\infty} \sum_{i=0}^n{{n}\choose{i}} \left(\left(-\frac{k}{n}\right)^i-\frac{i}{n}\left(-\frac{k}{n}\right)^{i-1}\right)&=e^{-k}-\lim_{n\to\infty} \sum_{i=0}^n \frac{i}{n}{{n}\choose{i}}\left(-\frac{k}{n}\right)^{i-1}\\&=e^{-k}-\lim_{n\to\infty} \sum_{i=1}^n {{n-1}\choose{i-1}}\left(-\frac{k}{n}\right)^{i-1}\\&=e^{-k}-\lim_{n\to\infty}\left(1-\frac{k}n\right)^{n-1}\\&=0.
\end{align*}

Thus, for any complete matching $M$ of $T$, we have
\begin{multline*}
\lim_{n\to\infty}  \left(\sum_{i=0}^n {{n}\choose{i}} \left(-\frac{k}{n}\right)^i\right)^{|U^c_{T,M}|}\left(\sum_{i=0}^n{{n}\choose{i}} \left(\left(-\frac{k}{n}\right)^i-\frac{i}{n}\left(-\frac{k}{n}\right)^{i-1}\right)\right)^{|U'_{T,M}|}\\=
\begin{cases}
e^{-k\cdot\inte(T,o)}&\text{if $M$ covers $V_{k-1}(B_{r-1}(T,o))$,}\\
0&\text{otherwise.}
\end{cases}
\end{multline*}

Thus, Lemma \ref{nehezlemma} follows.
\end{proof}

\subsection{The proof of Theorem \ref{thmFO}}

Recall that for any $o_n\in V_{k-1}(L_{n,k})$ and a semi-$k$-ary rooted tree $(T,o)$, we defined
\begin{multline*}
\mathcal{I}_{n,o_n}(T,o)=\\\{T_n|\text{ $T_n$ is an induced subgraph of $L_{n,k}$, $o_n\in V(T_n)$ and $(T,o)\cong (T_n,o_n)$ } \}.
\end{multline*}

Note that $|\mathcal{I}_{n,o_n}(T,o)|$ is the same for any choice of $o_n\in V_{k-1}(L_{n,k})$.

\begin{lemma}\label{izomszamol}
For any semi-$k$-ary rooted tree $(T,o)$,  we have
\[\lim_{n\to\infty}\frac{|\mathcal{I}_{n,[k]}(T,o)|}{n^{|V_k(T)|}}=\frac{(k!)^{|V_k(T)|}}{|\Aut(T,o)|}.\]
\end{lemma}
\begin{proof}
For any $k$-element subset $Y_n$ of $[n]$, an induced homomorphism from $(T,o)$ to $(L_{n,k},Y_n)$ is an injective map from  $\varphi:V(T)\to V(L_{n,k})$ such that $\varphi(o)=Y_n$ and $\varphi$ preserves both adjacency and non-adjacency. Let $\Ind(T,o;L_{n,k},Y_n)$ be the set of induced homomorphisms from $(T,o)$ to  $(L_{n,k},Y_n)$. Note that $|\Ind(T,o;L_{n,k},[k])|=|\Aut(T,o)|\cdot|\mathcal{I}_{n,[k]}(T,o)|$.

Consider an enumeration $v_1,v_2,\dots,v_{|V(T)|}$ of the vertices of $T$, such that if $i\le j$ then $d(o,v_i)\le d(o,v_j)$. For $v\in V(T)$, let $i(v)$ be the unique index such that~$v=v_{i(v)}$. 

Let $S_k$ be the set of permutations of $[k]$. We define a map $\psi=\psi_{n,Y_n}$ from $\Ind(T,o;L_{n,k},Y_n)$ to $[n]^{V_k(T)}\times S_k^{V_{k}(T)}$ in the following way. Take an element \break $\varphi\in \Ind(T,o;L_{n,k},Y_n)$, 
we define  $\psi(\varphi)=((x_v)_{v\in V_k(T)},(\sigma_v)_{v\in V_{k}(T)})$ as follows. 
For $v\in V_{k}(T)$, let $u_1,u_2,\dots, u_k$ be the children of $v$ ordered in such a way that\break $i(u_1)<i(u_2)<\dots<i(u_k)$. 
Let $p$ be the parent of $v$. Let $X=\varphi(v)$ and $Y=\varphi(p)=\{y_1<y_2<\dots<y_k\}$. 
We set $x_v$ to be the unique element of $X\backslash Y$. Moreover, let $\sigma_v$ be the unique element of $S_k$ such that for all $1\le j\le k$, we have $\varphi(u_j)=X\backslash\{y_{\sigma_v(j)}\}$.

We claim that the map $\psi$ is injective. Indeed, it can be easily proved by induction on $i$ that $\varphi(v_i)$ is uniquely determined by $\psi(\varphi)$. Along the lines of the proof of Claim~\ref{claim22}, one can check that if  $a=((x_v)_{v\in V_k(T)},(\sigma_v)_{v\in V_{k}(T)})\in [n]^{V_k(T)}\times S_k^{V_{k}(T)}$ is such that  $x_v$ ($v\in V_k(T))$ are pairwise distinct elements of $[n]\backslash Y_n$, then $a$ is in the image of $\psi$. Thus,
\[(k!)^{|V_k(T)|}\prod_{j=0}^{|V_k(T)|-1}(n-k-j)\le|\Ind(T,o;L_{n,k},Y_n)|\le (k!)^{|V_k(T)|} \cdot n^{|V_k(T)|}.\]
The lemma follows easily.
\end{proof}

Combining Lemma \ref{nehezlemma} and Lemma \ref{izomszamol}, we see that
\begin{align*}
\lim_{n\to\infty} &\mathbb{P}(B_r(G_{n,k},o_n) \cong (T,o))\\&=\lim_{n\to\infty}n^{-|V_k(T)|} \sum_{T_n\in \mathcal{I}_{n,o_n}(T,o)}n^{|V_k(T)|}\mathbb{P}(B_r(G_{n,k},o_n)=T_n)\\&=\lim_{n\to\infty}\frac{|\mathcal{I}_{n,o_n}(T,o)|}{n^{|V_k(T)|}}\cdot m^*(T,o)e^{-k\cdot \inte(T,o)}\\&=
\frac{m^*(T,o)\cdot(k!)^{|V_k(T)|}e^{-k|V_{k-1}(B_{r-1}(T,o))|}}{|\Aut(T,o)|}.
\end{align*} 
Thus, we proved Theorem \ref{thmFO}.

\section{The proof of Theorem \ref{quanched}}

We first prove that balls of $G_{n,k}$ have finite expected size.
\begin{lemma}\label{utszam}
Let $Y\in V_{k-1}(L_{n,k})$, then the number paths in $L_{n,k}$ of length $2\ell$ starting from $Y$ is at most
$(kn)^\ell.$    
\end{lemma}
\begin{proof}
Consider a path $P=(Y_0,X_1,Y_1,X_2,\dots,X_\ell,Y_\ell)$ such that $Y_0=Y$. For $i=1,2,\dots, \ell$, we define $x_i$ as the unique element of $X_i\backslash Y_{i-1}$. Moreover, let $Y_{i-1}=\{y_1<y_2<\dots <y_k\}$, we define $z_i$ as the index $j\in [k]$ such that $Y_{i}=X_i\backslash\{y_j\}$. It is straightforward to verify that the map $P\mapsto (x_1,z_1,\dots,x_k,z_k)$ is injective. Thus, the statement follows.  
\end{proof}

\begin{lemma}\label{finiteball}
For any even $r$, there is a constant $C_r=C_{r,k}$ such that for any $n$ and  $Y\in V_{k-1}(L_{n,k})$, we have
\[\mathbb{E}|V_{k-1}(B_r(G_{n,k},Y))|\le C_r.\]
\end{lemma}
\begin{proof}
Combining  Lemma \ref{utszam} and Lemma \ref{detbecs2}, 
we have
\[\mathbb{E}|V_{k-1}(B_r(G_{n,k},Y))|\le \sum_{\ell=0}^{r/2} (kn)^\ell \left(\frac{k+1}{n}\right)^\ell=\sum_{\ell=0}^{r/2} (k(k+1))^\ell.\] 
\end{proof}

To prove Theorem \ref{quanched}, we may assume that $k>1$ since the case $k=1$ is well known, see for example \cite[{Theorem~1.5}]{nachmias2020local}.

For any subgraph $G$ of $L_{n,k}$, we define 
\[\cup V(G)=\cup_{Z\in V(G)} Z.\] 

\begin{lemma}\label{disjunkt}
Let $Y_n$ and $Y_n'$ be two disjoint $k$ element subsets of $[n]$. Let $\mathcal{A}_n(Y_n,Y_n')$ be the event that $B_r(G_{n,k},Y_n)\cong (T,o)$, $B_r(G_{n,k},Y_n')\cong (T,o)$, and  $V(B_r(G_{n,k},Y_n))\cap V(B_r(G_{n,k},Y_n'))=\emptyset$. We have
\[\limsup_{n\to\infty}\quad \mathbb{P}(\mathcal{A}_n(Y_n,Y_n'))\le \left(\frac{m^*(T,o)\cdot(k!)^{|V_k(T)|}e^{-k(|V_{k-1}(B_{r-1}(T,o))|)}}{|\Aut(T,o)|}\right)^2.\]
\end{lemma}
\begin{proof}
Let us define \begin{multline*}
\mathcal{I}_{n,Y_n,Y_n'}(T,o)=\\\{(G,G')|\text{ $G$ and $G'$ are an induced subgraph of $L_{n,k}$, $Y_n\in V(G)$, $Y_n'\in V(G')$,}\\ \text{ $(T,o)\cong (G,Y_n)$, $(T,o)\cong (G,Y_n')$, $V(G)$ and $V(G')$ are disjoint} \},
\end{multline*}
and
\begin{multline*}
\bar{\mathcal{I}}_{n,Y_n,Y_n'}(T,o)=\\\{(G,G')|\text{ $G$ and $G'$ are an induced subgraph of $L_{n,k}$, $Y_n\in V(G)$, $Y_n'\in V(G')$,}\\ \text{ $(T,o)\cong (G,Y_n)$, $(T,o)\cong (G,Y_n')$, $\cup V(G)$ and $\cup V(G')$ are disjoint} \}.
\end{multline*}
\begin{lemma}\label{segedlemma}
We have
\[\lim_{n\to\infty}\frac{|\bar{\mathcal{I}}_{n,Y_n,Y_n'}(T,o)|}{n^{2|V_k(T)|}}=\left(\frac{(k!)^{|V_k(T)|}}{|\Aut(T,o)|}\right)^2.\]
and 
\[\lim_{n\to\infty}\frac{|{\mathcal{I}}_{n,Y_n,Y_n'}(T,o)\backslash \bar{\mathcal{I}}_{n,Y_n,Y_n'}(T,o)|}{n^{2|V_k(T)|}}=0.\]
\end{lemma}

\begin{proof}
We use the notations from Lemma \ref{izomszamol}. Let us consider the sets
\begin{multline*}{Q}_n=\{(\varphi,\varphi')\in \Ind(T,o;L_{n,k},Y_n)\times \Ind(T,o;L_{n,k},Y_n')\,:\,\varphi(V(T))\text{ and }\varphi'(V(T))\text{ are disjoint }\},
\end{multline*}
and
\begin{multline*}
\bar{Q}_n=\{(\varphi,\varphi')\in \Ind(T,o;L_{n,k},Y_n)\times \Ind(T,o;L_{n,k},Y_n')\,:\\ \cup_{v\in V(T)}\varphi(v)\text{ and }\cup_{v\in V(T)}\varphi'(v)\text{ are disjoint}\}.
\end{multline*}
Note that $|{Q}_n|=|\Aut(T,o)|^2 |{\mathcal{I}}_{n,Y_n,Y_n'}(T,o)|$ and $|\bar{Q}_n|=|\Aut(T,o)|^2 |\bar{\mathcal{I}}_{n,Y_n,Y_n'}(T,o)|$.

Observe that if
$a=((x_v)_{v\in V_k(T)},(\sigma_v)_{v\in V_{k}(T)})\in [n]^{V_k(T)}\times S_k^{V_{k}(T)}$ and \break $a'=((x'_v)_{v\in V_k(T)},(\sigma'_v)_{v\in V_{k}(T)})\in [n]^{V_k(T)}\times S_k^{V_{k}(T)}$ are such that  $x_v,x_v'$ ($v\in V_k(T))$ are pairwise distinct elements of $[n]\backslash(Y_n\cup Y_n')$, then there is a $(\varphi,\varphi')\in \bar{Q}_n$ such that $\psi_{n,Y_n}(\varphi)=a$ and $\psi_{n,Y_n'}(\varphi')=a'$.
Thus,
\begin{multline*}(k!)^{2|V_k(T)|}\prod_{j=0}^{2|V_k(T)|-1}(n-2k-j)\le|\bar{Q}_n|\\ \le|\Ind(T,o;L_{n,k},Y_n)|\cdot|\Ind(T,o;L_{n,k},Y_n')|\le (k!)^{2|V_k(T)|} \cdot n^{2|V_k(T)|}.
\end{multline*}
Therefore, the first statement follows easily.

To prove the second statement, observe that
\begin{multline*}|Q_n\backslash \bar{Q}_n|\le |\Ind(T,o;L_{n,k},Y_n)|\cdot|\Ind(T,o;L_{n,k},Y_n')|-|\bar{Q}_n|\\ \le (k!)^{2|V_k(T)|} \cdot n^{2|V_k(T)|}-(k!)^{2|V_k(T)|}\prod_{j=0}^{2|V_k(T)|-1}(n-2k-j)<Cn^{2|V_k(T)|-1}
\end{multline*}
for a large enough $C$.
\end{proof}
Let $(G,G')\in \bar{\mathcal{I}}_{n,Y_n,Y_n'}(T,o)$. Consider the set
\[E=\{Z\cup\{a\}|Z\in V_{k-1}(B_{r-2}(G',Y_n')), a\in \cup V(G)\}.\] 
Note that $E\subset V_{n,k}^{(G',Y_n')}\backslash V_k(G')$, and $|E|\le (k+1)|V_k(T,o)||V_{k-1}(T,o)|$.

We have
\begin{align*}\mathbb{P}(V_k(G_{n,k})&\cap (V_{n,k}^{(G',Y_n')}\backslash E)=V_k(G'))\\&\le \mathbb{P}(V_k(G_{n,k})\cap V_{n,k}^{(G',Y_n')}=V_k(G'))+\sum_{X\in E} \mathbb{P}(V_k(G')\cup \{X\}\subset V_k(G_{n,k}))\\
&\le \mathbb{P}(V_k(G_{n,k})\cap V_{n,k}^{(G',Y_n')}=V_k(G'))+|E| \left(\frac{k+1}{n}\right)^{|V_k(G')|+1}\\& \le 
\mathbb{P}(V_k(G_{n,k})\cap V_{n,k}^{(G',Y_n')}=V_k(G'))+C n^{-(|V_k(T)|+1)}
\end{align*} 
for some constant $C$ depending only on $(T,o)$ and $k$. At the second inequality, we used Lemma~\ref{detbecs2}.

For any $X\in V_{n,k}^{(G,Y_n)}$ and $X'\in V_{n,k}^{(G',Y_n')}\backslash E$, we have $|X\cap X'|\le 1$. Since we assumed that $k>1$, this implies that $P_{n,k}(X,X')=0$. Thus, from Lemma \ref{independent}, we see that the events $V_k(G_{n,k})\cap V_{n,k}^{(G,Y_n)}=V_k(G)$ and  $V_k(G_{n,k})\cap (V_{n,k}^{(G',Y_n')}\backslash E)=V_k(G')$ are independent. Using this, we have
\begin{align}\label{ineq23}
\mathbb{P}&(\text{$B_r(G_{n,k},Y_n)=G$ and $B_r(G_{n,k},Y_n')=G'$})\\&\le \mathbb{P}(\text{$V_k(G_{n,k})\cap V_{n,k}^{(G,Y_n)}=V_k(G)$ and  $V_k(G_{n,k})\cap (V_{n,k}^{(G',Y_n')}\backslash E)=V_k(G')$})\nonumber\\ &= \mathbb{P}(V_k(G_{n,k})\cap V_{n,k}^{(G,Y_n)}=V_k(G))\mathbb{P}(V_k(G_{n,k})\cap (V_{n,k}^{(G',Y_n')}\backslash E)=V_k(G'))\nonumber\\&\le \mathbb{P}(V_k(G_{n,k})\cap V_{n,k}^{(G,Y_n)}=V_k(G))\nonumber\\&\qquad\qquad\times\left(\mathbb{P}(V_k(G_{n,k})\cap V_{n,k}^{(G',Y_n')}=V_k(G'))+Cn^{-(|V_k(T)|+1)}\right)\nonumber\\&= \mathbb{P}(B_r(G_{n,k},Y_n)=G)\left(\mathbb{P}(B_r(G_{n,k},Y_n')=G')+Cn^{-(|V_k(T)|+1)}\right).\nonumber
\end{align}
    Let $\mathcal{B}_n(Y_n,Y_n')$ be the event that $B_r(G_{n,k},Y_n)\cong (T,o)$, $B_r(G_{n,k},Y_n')\cong (T,o)$, and $\cup V(B_r(G_{n,k},Y_n))\cap \cup V(B_r(G_{n,k},Y_n'))=\emptyset$. Therefore, 
\begin{align*}
\limsup_{n\to\infty}&\quad \mathbb{P}(\mathcal{B}_n(Y_n,Y_n'))\\&=\limsup_{n\to\infty} \sum_{(G,G')\in \bar{\mathcal{I}}_{n,Y_n,Y_n'}(T,o)} \mathbb{P}(\text{$B_r(G_{n,k},Y_n)=G$ and $B_r(G_{n,k},Y_n')=G'$})\\ &\le \limsup_{n\to\infty} \sum_{(G,G')\in \bar{\mathcal{I}}_{n,Y_n,Y_n'}(T,o)} \mathbb{P}(B_r(G_{n,k},Y_n)=G)\\ &\qquad\qquad\qquad\qquad\qquad\qquad\qquad\times\left(\mathbb{P}(B_r(G_{n,k},Y_n')=G')+Cn^{-(|V_k(T)|+1)}\right)\\ &=\limsup_{n\to\infty} \frac{|\bar{\mathcal{I}}_{n,Y_n,Y_n'}(T,o)|}{n^{2|V_k(T)|}}\left(m^*(T,o)\cdot e^{-k(|V_{k-1}(B_{r-1}(T,o))|)}\right)^2\\ &=\left(\frac{m^*(T,o)\cdot(k!)^{|V_k(T)|}e^{-k(|V_{k-1}(B_{r-1}(T,o))|)}}{|\Aut(T,o)|}\right)^2.
\end{align*}
where at the inequality, we used Inequality \eqref{ineq23}; at the second to last equality, we used Lemma~\ref{nehezlemma}; and at the last equality, we used Lemma \ref{segedlemma}.

Moreover, combining Lemma \ref{segedlemma} and Lemma \ref{detbecs2}, we see that
\begin{align*}
&\lim_{n\to\infty} \left(\mathbb{P}(\mathcal{A}_n(Y_n,Y_n'))-\mathbb{P}(\mathcal{B}_n(Y_n,Y_n'))\right)\\&=\lim_{n\to\infty} \sum_{(G,G')\in \mathcal{I}_{n,Y_n,Y_n'}(T,o)\backslash \bar{\mathcal{I}}_{n,Y_n,Y_n'}(T,o)} \mathbb{P}(\text{$B_r(G_{n,k},Y_n)=G$ and $B_r(G_{n,k},Y_n')=G'$})\\&\le \lim_{n\to\infty} |\mathcal{I}_{n,Y_n,Y_n'}(T,o)\backslash \bar{\mathcal{I}}_{n,Y_n,Y_n'}(T,o)| \left(\frac{k+1}{n}\right)^{2|V_k(T)|} \\&=0.
\end{align*}
Thus, Lemma \ref{disjunkt} follows.
\end{proof}

It is clear from Theorem \ref{thmFO} that 
\[\lim_{n\to\infty} \mathbb{E} \frac{C_n}{{{n}\choose{k}}}=\frac{m^*(T,o)\cdot(k!)^{|V_k(T)|}e^{-k|V_{k-1}(B_{r-1}(T,o))|}}{|\Aut(T,o)|}.\]

Let us estimate the second moment of $C_n$:

\begin{align*}
&\limsup_{n\to \infty} \quad \mathbb{E}\frac{C_n^2}{{{n}\choose{k}}^2}\\
&=\limsup_{n\to \infty}\frac{1}{{{n}\choose{k}}^2}\\&\qquad\sum_{Y\in V_{k-1}(L_{n,k})}\sum_{Y'\in V_{k-1}(L_{n,k})} \mathbb{P}(\text{$B_r(G_{n,k},Y_n)\cong (T,o)$ and $B_r(G_{n,k},Y_n')\cong (T,o)$})\\
&\le \limsup_{n\to \infty}\frac{1}{{{n}\choose{k}}^2} \sum_{Y\in V_{k-1}(L_{n,k})}\Big(|\{Y'\in V_{k-1}(L_{n,k})|Y\cap Y'\neq \emptyset\}|\\
&\qquad\qquad+\sum_{Y'\in V_{k-1}(L_{n,k})} \mathbb{P}(B_r(G_{n,k},Y)\text{ and  }B_r(G_{n,k},Y')\text{ are not disjoint})\\
&\qquad\qquad+\sum_{\substack{Y'\in V_{k-1}(L_{n,k})\\Y\cap Y'=\emptyset}} \mathbb{P}(\mathcal{A}_n(Y,Y'))\Big)\\
&\le\limsup_{n\to \infty}\frac{1}{{{n}\choose{k}}^2} \sum_{Y\in V_{k-1}(L_{n,k})}\Bigg(\mathbb{E}|V_{k-1}(B_{2r}(G_{n,k},Y))|\\
&\qquad\qquad+\sum_{\substack{Y'\in V_{k-1}(L_{n,k})\\Y\cap Y'=\emptyset}} \left(\frac{m^*(T,o)\cdot(k!)^{|V_k(T)|}e^{-k|V_{k-1}(B_{r-1}(T,o))|}}{|\Aut(T,o)|}\right)^2\Bigg)\\
&=\left(\frac{m^*(T,o)\cdot(k!)^{|V_k(T)|}e^{-k|V_{k-1}(B_{r-1}(T,o))|}}{|\Aut(T,o)|}\right)^2\\
&=\left(\lim_{n\to\infty}\mathbb{E}\frac{C_n}{{{n}\choose{k}}}\right)^2,
\end{align*}
where we used Lemma \ref{disjunkt} and Lemma \ref{finiteball}.

Therefore,
\[\lim_{n\to\infty} \mathrm{Var}\left(\frac{C_n}{{{n}\choose{k}}}\right)=0.\]
Thus, Theorem \ref{quanched} follows from Chebyshev's inequality.


\section{The proof of Lemma \ref{generate}}

A plane tree is a tuple $(T,o,(<_v)_{v\in V(T)})$, where $(T,o)$ is a rooted tree, $<_v$ is a total order on the children of $v$. A plane tree isomorphism is a graph isomorphism which preserves the root and also preserves the total ordering at each vertex. Note that if there is a plane tree isomorphism between two plane trees, then it must be unique. The tree $(\mathbb{T}_k,o)$ produced by Algorithm \ref{alg1} can be considered as a plane tree if we assume that $u_1^v <_v u_2^v <_v \dots<_v u_{c_v}^v$. 

Let $(T,o,(<_v)_{v\in V(T)})$ be a plane tree such that $(T,o)$ is a semi-$k$-ary tree of depth $r$. Let $M$ be a matching that covers all the vertices of $B_{r-1}(T,o)$. For any rooted tree $(T,o)$ and $v\in V(T)$, let $c_v$ be the number of children of $v$.

\begin{lemma}  
 Consider the event that there is a plane tree isomorphism $\psi$ from $(T,o,(<_v)_{v\in V(T)})$ to 
$B_r(\mathbb{T}_k,o,(<_v)_{v\in V(\mathbb{T}_k)})$ such that $\psi(M)=\mathbb{M}_k\cap B_r(\mathbb{T}_k,o)$. This event has probability
\[\exp(-k |V_{k-1}(B_{r-1}(T,o))|) \prod_{v\in V_{k-1}(B_{r-1}(T,o))} \frac{1}{c_v!}.\]
\end{lemma}
\begin{proof}
Let 
\begin{align*}D_k&=\{v\in V_k(T,o) | v\text{ is matched to one of its children by }M\},\\
D_{k-1}&=\{v\in V_{k-1}(B_{r-1}(T,o)) | v\text{ is matched to one of its children by }M\} \text{ and}\\
U_{k-1}&=\{v\in V_{k-1}(B_{r-1}((T,o)) | v\text{ is matched to its parent by }M\}.
\end{align*}

Let $p(i)=\frac{k^i\exp(-k)}{i!}$. Note that $|D_{k-1}|+|D_{k}|=|M|=|V_{k}(T,o)|$, and \[\sum_{v\in V_{k-1}(B_{r-1}(T,o))} c_v=|V_{k}(T,o)|.\]

  It is easy to see that the event in the statement of the lemma has probability
\begin{align*}
&\frac{1}{k^{|D_{k}|}}\left(\prod_{v\in D_{k-1}} p(c_v-1)\frac{1}{c_v}\right) \left(\prod_{v\in U_{k-1}} p(c_v)\right)\\
&=\frac{1}{k^{|D_{k}|}}\left(\prod_{v\in D_{k-1}} p(c_v)\frac{1}{k}\right) \left(\prod_{v\in U_{k-1}} p(c_v)\right)\\
&=\frac{1}{k^{|D_{k}|+|D_{k-1}|}}\prod_{v\in V_{k-1}(B_{r-1}(T,o))} p(c_v)\\
&=\frac{1}{k^{|V_{k}(T,o)|}} k^{\sum_{v\in V_{k-1}(B_{r-1}(T,o))}c_v} \exp(-k |V_{k-1}(B_{r-1}(T,o))|) \prod_{v\in V_{k-1}(B_{r-1}(T,o))} \frac{1}{c_v!}\\
&=\exp(-k |V_{k-1}(B_{r-1}(T,o))|) \prod_{v\in V_{k-1}(B_{r-1}(T,o))} \frac{1}{c_v!}.
\end{align*}
\end{proof}
As an immediate corollary of the previous lemma, we get the following lemma.

\begin{lemma}\label{planecount}  
 Consider the event that there is a plane tree isomorphism  from $(T,o,(<_v)_{v\in V(T)})$ to 
$B_r(\mathbb{T}_k,o,(<_v)_{v\in V(\mathbb{T}_k)})$. This event has probability
\[m^*(T,o)\exp(-k |V_{k-1}(B_{r-1}(T,o))|) \prod_{v\in V_{k-1}(B_{r-1}(T,o))} \frac{1}{c_v!}.\]
\end{lemma}

The proof of the next proposition is straightforward.
\begin{proposition}\label{pcount}
Given a rooted semi-$k$-ary tree $(T,o)$,  up to isomorphisms there are
\[\frac{(k!)^{|V_k(T,o)|}\prod_{v\in V_{k-1}(B_{r-1}(T,o))} c_v!}{|\Aut(T,o)|} \]
ways to turn it into plane tree.
\end{proposition}

Lemma \ref{generate} follows by combining Lemma \ref{planecount} and Proposition \ref{pcount}.

\section{Open questions}

\begin{problem}
Is there a local weak limit theorem if we consider the uniform measure on $\mathcal{C}(n,k)$ instead of $\nu_{n,k}$?
\end{problem}

Note that the asymptotic cardinality of the set $\mathcal{C}(n,k)$ is already a difficult question, see the work of Linial and Peled \cite{linial2019enumeration}.

The following question was already asked by Lyons \cite{lyons2003determinantal}.
\begin{problem}
Is there a way to sample from $\nu_{n,k}$ which is analogous to the Aldous-Broder  algorithm \cite{aldous1990random,broder1989generating} or Wilson's algorithm \cite{wilson1996generating}? 
\end{problem}
Note that there is a general polynomial algorithm \cite{hough2006determinantal} to sample from determinantal processes. However, this algorithm is not very fast, and gives no real insight about the structure of the sample.

We also repeat a few open question from the paper of Kahle and Newman \cite{kahle2020topology}, where certain topological properties of random simplicial complexes with law $\nu_{n,3}$ were investigated.
\begin{problem}
Is there a scaling limit of simplicial complexes of law $\nu_{n,k}$?
\end{problem}  
Note that for $k=1$, various scaling limits were obtained by Aldous \cite{aldous1991continuumI,aldous1991continuumII,aldous1993continuumIII}.

\begin{problem}
Let $C_n$ be a random simplicial complex with law $\nu_{n,k}$. Does $H_{k-1}(C_n)$ follow the Cohen-Lenstra heuristics? In other words, is it true that given a prime $p$ and a finite abelian $p$-group $G$, the
probability that the $p$-Sylow subgroup of $H_{k-1}(C_n)$ is isomorphic to $G$ tends to
\[\frac{\prod_{i=1}^{\infty}(1-p^{-i})}{|\Aut(G)|}?\]
The same question is also interesting for the uniform measure on $\mathcal{C}(n,k)$.  
\end{problem}

Experimental results in \cite{kahle2020cohen,kahle2020topology} suggest that the answer should be affirmative for both models. Note the Cohen-Lenstra heuristics first appeared in a number theoretic setting \cite{cl1}. But since then, these heuristics and their modified versions also appeared in more combinatorial settings \cite{cl2,cl3,cl4,cl5,cl6}.

\bibliography{references}
\bibliographystyle{plain}



\end{document}